\documentclass[english]{amsart}

\usepackage{babel}
\usepackage{amstext}
\usepackage{amsmath}
\usepackage{amsfonts}
\usepackage{latexsym}
\usepackage{ifthen}
\newcommand{\AN}{{\mathbb A}}

\newcommand\bR{{\mathbb R}}
\newcommand\bZ{{\mathbb Z}}
\newcommand\bC{{\mathbb C}}
\newcommand\bQ{{\mathbb Q}}

\newcommand\bP{{\mathbb P}}

\newcounter{lemma}

\usepackage{enumerate}
\newenvironment{enumi}{\begin{enumerate}[\upshape \hspace{0.25cm} (i)]}{\end{enumerate}}

\newenvironment{enumu}{\begin{enumerate}[\upshape \hspace{0.5cm} 1.]}{\end{enumerate}}

\theoremstyle{plain} 
\newtheorem{theorem}{\noindent\bf Theorem}[section]
\newtheorem{lemma}[theorem]{\noindent\bf Lemma}

\newtheorem{proposition}[theorem]{\noindent\bf Proposition}

\newtheorem{definition}[theorem]{\noindent\bf Definition}

\theoremstyle{definition} 
\newtheorem{remark}[theorem]{\noindent\bf Remark}
\newtheorem{example}[theorem]{\noindent\bf Example}
\newtheorem{notation}[theorem]{\noindent\bf Notation}


\title[Automorphisms of rational manifolds] {Automorphisms of rational manifolds of positive entropy 
with Siegel disks} 
\author{Keiji Oguiso and Fabio Perroni}

\begin{document}

\begin{abstract}
Using McMullen's rational surface automorphisms, 
we construct projective rational manifolds 
of higher dimension admitting automorphisms of positive entropy 
with arbitrarily high number of Siegel disks and those with exactly 
one Siegel disk.
\end{abstract}

\maketitle


\section{Introduction}
\noindent
In his beautiful paper \cite{Mc07}, McMullen constructed rational surface automorphisms of positive entropy with Siegel disks. They are the first examples among automorphisms of projective manifolds. (See also \cite{BK06}, \cite{BK09}.)  From one side, positive entropy indicates that general orbits spread out vastly even though the initial points are very close, from which one might expect that the general orbit could be densely distributed. But on the other side, the existence of Siegel disks shows that there is no dense orbit and the orbit of any point in the disk never goes out of the disk. This contrast makes the study of automorphisms of manifolds of positive entropy with Siegel disks very attractive. In McMullen's construction, the automorphism has exactly one Siegel disk and it is 
arithmetic (see also Section 2 for definitions and more details). 
It is then natural to ask:

(i) How about in higher dimension?

(ii) How many Siegel disks can an automorphism of positive entropy have? 
\par
\vskip 4pt
The aim of this paper is to address these questions: 

\begin{theorem}\label{main} 
\begin{enumu}
\item[]
\item Let $n$ be any integer such that $n \ge 4$ and $N$ be an arbitrary positive integer. Then, for each such $n$ and $N$, 
there is a pair $(X, g)$ of a non-singular complex projective rational variety 
$X$ of dimension $n$ and an automorphism $g \in {\rm Aut}\, (X)$ such that:\\
(1-i) the entropy $h(g)$ of $g$ is positive; and\\
(1-ii) $g$ admits at least $N$ Siegel disks and they are 
arithmetic.

\item There is a pair $(X, g)$ of a non-singular complex projective rational $3$-fold $X$ and an automorphism 
$g\in {\rm Aut}\,(X)$ such that:\\
(2-i) the entropy $h(g)$ of $g$ is positive; and\\
(2-ii) $g$ admits exactly $2$ Siegel disks and they are 
arithmetic.  
\item Let $n$ be any even integer such that $n \ge 4$. Then, for each such $n$, 
there is a pair $(X, g)$ of a non-singular complex projective rational variety $X$ of dimension 
$n$ and an automorphism 
$g\in {\rm Aut}\,(X)$ such that:\\
(3-i) the entropy $h(g)$ of $g$ is positive; and\\
(3-ii) $g$ admits exactly one Siegel disk and it is 
arithmetic.
\end{enumu}  
\end{theorem} 

We believe that this theorem gives the first examples of automorphisms of 
projective manifolds of positive entropy with Siegel disks 
{\it in dimension} $\ge 3$. Here, it is essential to make the entropy positive. Indeed, there are a lot of automorphisms of $\bP^n$, of which 
the entropy is necessarily $0$, having (arithmetic) Siegel disks. Our construction is the product construction made of McMullen's rational surfaces, projective toric manifolds and their automorphisms. More explicit statements are given in Section 3 (Theorems (\ref{main1}) and (\ref{main2})). In this sense, our construction of manifolds are rather easy modulo McMullen's deep construction. On the other hand, we should also note that for a manifold $S$ and its automorphism $g$ with a fixed point $P$, the product automorphism $g \times g$ of $S \times S$ has no Siegel disk at $(P, P)$, even if $g$ itself has a Siegel disk at $P$. So, the essential point in the product construction is to {\it choose} manifolds and their automorphisms so that the eigenvalues of the product 
action at the fixed point are multiplicatively independent within the algebraic integers of absolute value $1$. This turns out to be a kind of arithmetic problem which has  its own interest. The precise formulation is given in Section 4, Definition (\ref{mau}), and
its solution is contained in Theorem (\ref{exist}). 
\par
\vskip 4pt
Throughout this note, we work over the field of  complex numbers $\bC$.


\section{McMullen's pair}
\setcounter{lemma}{0}
\noindent
In this section we review McMullen's rational surface automorphisms together with some relevant notions. 
\par
\vskip 4pt
\noindent
{\it (i) Entropy.} 
Let $X$ be a compact metric space with distance function $d$. Let $g$ be a continuous surjective self map of $X$. Roughly 
speaking, the {\it entropy} is a 
measure of ``how fast two orbits $\{g^{k}(x)\}_{k \ge 0}$, $\{g^{k}(y)\}_{k \ge 0}$ spread out when $k \rightarrow \infty$''. 
We recall here its definition and a characterization in cohomological terms that will be used later.
For more details we refer to \cite{KH95}.
For any   $n \in \bZ_{>0}$, consider the metric 
$$
d_{g, n}(x, y) := {\rm max}\, \{d(g^{k}(x), g^{k}(y))\, \vert\, 0 \le k \le n-1\}\, .
$$
The {\it entropy} of $g$ is then defined as (\cite{KH95}, Page 108, formula (3.1.10)):
$$
h(g) := {\rm lim}_{\epsilon \rightarrow 0}\, {\rm limsup}_{n \rightarrow \infty}\, \frac{\log S(g, \epsilon, n)}{n}\, .
$$
Here $S(g, \epsilon, n)$ is the {\it minimal} number of $\epsilon$-balls, with respect to $d_{g, n}$, that cover $X$. It is shown that $h(g)$ does not depend on the choice of the distance $d$ giving the same topology on $X$ (see e.g. \cite{KH95}, Page 109, Proposition (3.1.2)).
From this definition it is easy to grasp the meaning of the entropy. 
However, for our computations, the following fundamental theorem, due to Gromov-Yomdin-Friedland (\cite{Fr95}, Theorem (2.1)),
will be more convenient:

\begin{theorem}\label{gyf} 
Let $X$ be a compact K\"ahler manifold of dimension $n$ 
and let \\
$g : X \longrightarrow X$ 
be a holomorphic surjective self map of $X$. Then 
$$h(g) = \log \rho(g^{*} \vert \oplus_{k=0}^{n} H^{2k}(X, \bZ))\, .$$
Here $\rho(g^{*} \vert \oplus_{k=0}^{n} H^{2k}(X, \bZ))$ is 
the spectral radius 
of the action of $g^{*}$ on the total cohomology ring of even degree. 
In particular, $h(g)$ is the logarithm of an algebraic integer.
\end{theorem}

See also \cite{Zh08} for some role of the entropy in the classification of higher dimensional varieties.

\newpage
\noindent
{\it (ii) Salem polynomials and Salem numbers.} 
\begin{definition}\label{defsalem}
A Salem polynomial is a monic irreducible reciprocal polynomial $\varphi(x)$ in $\bZ[x]$ such that 
$$
\{x\in \bC\,|\, \varphi(x)=0\}=\{\eta\, , \, \frac{1}{\eta} \, , \, \delta_1 \, , \, \overline{\delta_1}\, ,
\, \dots \, , \, \, \delta_{n-1} \, , \, \overline{\delta_{n-1}}\, \}\, ,
$$
where $|\delta_i|=1$ and $\eta>1$ is real.
Notice that $\varphi(x)$ is necessarily of even degree.

A Salem number is the unique real root $\eta > 1$. In other words, a Salem number of degree $2n$ is a real algebraic integer $\eta >1$ whose Galois conjugates 
consist of $1/\eta$ and $2n-2$ imaginary numbers on $S^{1}=\{z\in \bC\,|\, |z|=1\, \}$. 
\end{definition}

Let $\varphi_{2n}(x)$ be a Salem polynomial of degree $2n$. As $\varphi_{2n}(x)$ 
is monic irreducible and reciprocal, there is a unique monic irreducible polynomial 
$r_n(x) \in \bZ[x]$ of degree $n$ such that
$$\varphi_{2n}(x) = x^n \cdot r_n(x + \frac{1}{x})\, .$$
We call this polynomial $r_n(x)$ the {\it Salem trace polynomial} of 
$\varphi_{2n}(x)$. If
$$\eta\, ,\, \frac{1}{\eta}\, ,\, \delta_{i}\, ,\, \overline{\delta_{i}} = \frac{1}{\delta_i}\,\, 
(1 \le i \le n-1)$$
are the roots of $\varphi_{2n}(x) = 0$, then 
the roots of $r_n(x) = 0$ are:
$$\eta + \frac{1}{\eta}\, ,\, \delta_i + \frac{1}{\delta_i} = \delta_{i} + 
\overline{\delta_{i}}\,\, (1 \le i \le n-1)\, .$$
\par
\vskip 4pt
\noindent
{\it (iii) Coxeter element.} By $E_n(-1)$ we denote the lattice represented by the Dynkin diagram with $n$ vertices $s_k$ ($0 \le k \le n-1$) 
of self-intersection $-2$ such that $n-1$ vertices $s_1$, $s_2$, $\cdots$, $s_{n-1}$ form a Dynkin diagram of type $A_{n-1}(-1)$ in this order and the remaining vertex $s_0$ joins to only the vertex $s_3$ by a simple line, as shown in Figure 1.

\begin{center}
\setlength{\unitlength}{0.8cm}
\begin{picture}(8,3)
\thicklines
\put(0,1){\circle{0.2}}
\put(-0.15,1.5){\makebox{$s_1$}}
\put(0.1,1){\line(1,0){0.8}}
\put(1,1){\circle{0.2}}
\put(0.85,1.5){\makebox{$s_2$}}
\put(1.1,1){\line(1,0){0.8}}
\put(2,1){\circle{0.2}}
\put(1.85,1.5){\makebox{$s_3$}}
\put(2,0.1){\line(0,1){0.76}}
\put(2,0){\circle{0.2}}
\put(2.2,-0.1){\makebox{$s_0$}}
\put(2.1,1){\line(1,0){0.8}}
\put(3,1){\circle{0.2}}
\put(2.85,1.5){\makebox{$s_4$}}
\put(3.1,1){\line(1,0){0.8}}
\put(4,1){\circle{0.2}}
\put(3.85,1.5){\makebox{$s_5$}}
\put(4.1,1){\line(1,0){0.4}}
\put(4.7,0.99){\makebox{$\dots$}}
\put(5.5,1){\line(1,0){0.4}}
\put(6,1){\circle{0.2}}
\put(5.7,1.5){\makebox{$s_{n-1}$}}
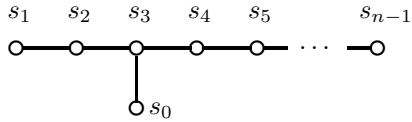
\put(0.5,-1.4){\makebox{Figure 1. The $E_n(-1)$ diagram.}}
\end{picture}
\end{center}

\vspace{1.5cm}

\noindent The lattice $E_n(-1)$ is of signature $(1, n-1)$, when $n \ge 10$. 

Let $W(E_{n}(-1))$ be the Weyl group of $E_n(-1)$, i.e., the 
subgroup of ${\rm O}(E_n(-1))$ generated by the reflections
$$
r_k (x) = x + (x,s_k)s_k\, .
$$  
The Weyl group $W(E_n(-1))$ has a special conjugacy class called 
the {\it Coxeter class}. It is the conjugacy class of 
the product (in any order in our case) of the reflections 
$$w_n := \Pi_{k=0}^{n-1} r_k\, .$$ 
The following theorem follows from either \cite{BK09}, Theorem 3.3
or \cite{GMH08},\\ Theorem (1.1), Corollary (1.2). We follow the notation of \cite{GMH08}:
\begin{theorem}\label{cox} Let $E_n(x)$ be the characteristic polynomial 
of the Coxeter element $w_n$. 
Then, for $n \ge 10$, 
$$E_n(x) = C_n(x)\varphi(x)$$
where $C_n(x)$ is the product of the cyclotomic factors and $\varphi(x)$ 
is a Salem polynomial. Moreover, $C_n(x) = C_m(x)$ if $n \equiv m\, {\rm mod}\, 360$.  
\end{theorem}

By \cite{GMH08}, Corollary 4.3, we have
\begin{equation}\label{E_n(x)}
E_n(x)(x-1)=x^{n-2}(x^3-x-1)+(x^3+x^2-1)\, ,
\end{equation} 
hence the formula in Theorem (\ref{cox}) can be used to determine 
both $C_n(x)$ and $\varphi(x)$.
 
The following example, for $n=19$, will be used  later:
\begin{example}\label{19}
$$E_{19}(x) = (x+1)(x^4+x^3+x^2+x+1)\varphi_{14}(x)\, .$$
Here $\varphi_{14}(x)$ is a Salem polynomial of degree $14$:
$$\varphi_{14}(x) = x^{14} -x^{13} -x^{11} +x^{10} -x^{7} +x^{4} -x^{3} -x +1\, .$$
\end{example}
\par
\vskip 4pt
\noindent
{\it (iv) Siegel disks and arithmetic Siegel disks.}
\begin{definition}\label{sd}

(1) Let $\Delta^n$ be an $n$-dimensional unit disk 
with linear coordinates 
$$(z_1, z_2, \dots , z_n)\, .$$ 
A linear automorphism (written under the coordinate action)
$$f^{*}(z_1, z_2, \dots, z_n) = (\rho_1z_1, \rho_2z_2, 
\dots, \rho_nz_n)$$
is called an irrational rotation if 
$$\vert \rho_1 \vert = \vert \rho_2 \vert = \cdots = \vert \rho_n \vert =1\, ,$$
and $\rho_1$, $\rho_2$, $\dots$, $\rho_n$ are multiplicatively independent, i.e. 
$$(m_1, m_2, \dots , m_n) = (0, 0, \dots , 0)$$ 
is the only integer solution to 
$$\rho_1^{m_1}\rho_2^{m_2}\cdots \rho_n^{m_n} = 1\, .$$

(2) Let $X$ be a complex manifold of dimension 
$n$ and $g$ be an automorphism of $X$. A domain $U \subset X$ is called 
a Siegel disk of $(X, g)$ if $g(U) =U$ and $(U, g \vert U)$ 
is isomorphic to some irrational rotation $(\Delta^n, f)$. In other words, 
$g$ has a Siegel disk if and only if there is a fixed point $P$ at which 
$g$ is locally 
analytically linearized as in the form of an irrational rotation. We call 
the Siegel disk arithmetic if in addition all $\rho_i$ are algebraic integers.
\end{definition}

The first examples of 
surface automorphisms with Siegel 
disks were discovered by McMullen (\cite{Mc02}, Theorem (1.1)) within 
K3 surfaces. See also \cite{Og09}, Theorem (1.1) for a similar example. 
The resultant K3 surfaces $X$ 
are necessarily of 
algebraic dimension $0$ (\cite{Mc02}, Theorem (3.5), see also 
\cite{Og08}, Theorem (2.4)). Later, 
McMullen (\cite{Mc07}, Theorem (10.1)) found rational surface automorphisms with arithmetic Siegel disks. 
\par
\vskip 4pt
\noindent 
{\it (v) McMullen's pair.} Let $S$ be a blowup of $\bP^2$ at $n$ (distinct) points. Then, $H^2(S, \bZ)$ is isomorphic to the odd unimodular lattice of signature $(1,n)$. The orthogonal complement $(-K_S)^{\perp}$ 
is then isomorphic to $E_n(-1)$ and ${\rm Aut}\,(S)$ naturally acts on 
$E_n(-1)$ (under a fixed marking). As a part of more general results, 
McMullen proves the following theorem (See \cite{Mc07}, Theorem (10.1), see also Theorem (10.3), proof of Theorem (10.4) and the formula (9.1)):

\begin{theorem}\label{Mc} Let $n$ be a sufficiently large integer such that 
$n \equiv 1\, {\rm mod}\, 6$. Then, for each such $n$, 
there are a rational surface $S =S(n)$ which is a blow up of $\bP^2$ at $n$ distinct points and an automorphism $F = F(n)$ such that:
\begin{enumu}
\item The characteristic polynomial of $F^{*} \vert H^2(S, \bZ)$ is
$$E_{n}(x)(x-1) =(x-1) C_n(x)\varphi(x)$$
where $E_{n}(x)$ is the characteristic polynomial of the Coxeter element 
of $E_n(-1)$, $C_n(x)$ is the product of cyclotomic polynomials
and $\varphi(x)$ is a Salem polynomial. 

\item The fixed point set $S^F$ consists of exactly $2$ points, say, $P$ and $Q$. 
Moreover, $F$ has an arithmetic Siegel disk at $Q$ but $F$ has no Siegel disk 
at $P$ (in fact the eigenvalues of $F^{*} \vert T_{S, P}^{*}$ 
are not multiplicatively independent). 

\item Let 
$$F^{*}(x, y) = (\alpha(n)x, \beta(n)y)$$ 
be the locally analytic linearization of $F$ at $Q$. So, $\alpha(n)$ 
and $\beta(n)$ are multiplicatively independent and of absolute value 
$1$. Then, there is a root $\delta(n)$ of $\varphi(x) =0$ of absolute 
value $1$ such that $\alpha(n)$ and $\beta(n)$ 
satisfy 
$$\alpha(n)\beta(n) = \delta(n)\, ,\, 2 + \frac{\alpha(n)}{\beta(n)} 
+ \frac{\beta(n)}{\alpha(n)} = \frac{\delta(n)(1+\delta(n))^2}
{(1 + \delta(n) + \delta(n)^{2})^{2}}\,.$$
In particular, $\alpha(n)^2$ and $\beta(n)^2$ are the roots of the quadratic 
equation of the form
$$x^2 + a(\delta(n))x + \delta(n)^2 = 0$$
where $a(x) \in \bQ(x)$. 

\item There are another root $\delta'(n)$ of $\varphi(x) =0$ of absolute 
value $1$ 
and complex numbers $\alpha'(n)$, $\beta'(n)$ such that
$$\vert \alpha'(n)/\beta'(n) \vert \not= 1\, ,$$
$$\alpha'(n)\beta'(n) = \delta'(n)\, ,\, 2 + \frac{\alpha'(n)}{\beta'(n)} 
+ \frac{\beta'(n)}{\alpha'(n)} = \frac{\delta'(n)(1+\delta'(n))^2}
{(1 + \delta'(n) + \delta'(n)^{2})^{2}}\,.$$
In particular, $\alpha'(n)^2$ and $\beta'(n)^2$ are the roots of the quadratic 
equation
$$x^2 + a(\delta'(n))x + \delta'(n)^{2} = 0\, .$$
\end{enumu}
\end{theorem}
We call any pair $(S, F)$ as in Theorem (\ref{Mc}) a {\it McMullen's pair}. 
Notice that, by Theorem (\ref{Mc}) (1),  any McMullen's pair is of positive entropy.

\medskip

In \cite{Mc07}, it is not explicit that $\alpha(n)^{\pm 1}$ and $\beta(n)^{\pm 1}$
are algebraic integers. Since we will need this fact, we give here a proof.

\begin{proposition}\label{2.7}
Let $\alpha(n)$ and $\beta(n)$ be as in Theorem (\ref{Mc}).
Then $\alpha(n)$, $\beta(n)$, $\alpha(n)^{-1}$ and $\beta(n)^{-1}$ are algebraic integers.
\end{proposition}
\begin{proof}
Set $\alpha:=\alpha(n),\, \beta:=\beta(n)$ and $\delta:=\delta(n)$.
We first prove the Proposition for $\alpha$ and $\beta$.
 
From the previous Theorem \eqref{Mc} (3), we know that $\delta=\alpha \beta$ is an algebraic integer, 
therefore it is enough to show that $\alpha+\beta$ is so. 
From the equation: 
$$
\frac{(\alpha +\beta)^2}{\alpha\beta}=
\frac{\delta(1+\delta)^2}{(1+\delta+\delta^2)^2}\, ,
$$
we have that 
$$
\alpha+\beta=\pm \frac{\delta(1+\delta)}{1+\delta+\delta^2}\, ,
$$
hence we only need to prove that $\frac{1}{1+\delta+\delta^2}$ is an algebraic integer.
We use now formula \eqref{E_n(x)} for the characteristic polynomial $E_n(x)$ of the 
Coxeter element $w_n\,$. Since $n=6k+1$, it readily follows from \eqref{E_n(x)} that
there exists $A(x)\in \mathbb{Z}[x]$ such that
\begin{equation}\label{EQ}
E_n(x)(x-1)=(x^2+x+1)A(x) - (x+2)\, .
\end{equation}
Since $E_n(\delta)=0\, $, we have:
\begin{equation}\label{Qd}
\frac{A(\delta)}{\delta+2}=\frac{1}{\delta^2+\delta+1}\, .
\end{equation}
On the other hand, we can write 
$$
\frac{A(\delta)}{\delta+2}=B(\delta)+\frac{A(-2)}{\delta+2}\, , \qquad \mbox{for some}\quad B(x)\in \mathbb{Z}[x]\, .
$$ 
Hence, it is enough to prove that $\frac{A(-2)}{\delta+2}$ is an algebraic integer.
From \eqref{EQ} it follows that $E_n(-2)=-A(-2)$, therefore, there exists $C(x)\in \mathbb{Z}[x]$
such that 
$$
E_n(x)=(x+2)C(x)-A(-2)\, .
$$
We conclude that 
$$
\frac{A(-2)}{\delta+2}=C(\delta)\, ,
$$
which is an algebraic integer, thus $\alpha +\beta$ is an algebraic integer.

The statement for $\alpha^{-1}$ and $\beta^{-1}$ follows by replacing $F$ with $F^{-1}$
in Theorem \eqref{Mc}.
\end{proof}

\section{Statement of the main results}
\setcounter{lemma}{0}
\noindent
In this section we state our main results more explicitly. 
\begin{notation}\label{nottion}
In the following, we denote by $Y_{\Delta}$ the $d$-dimensional complete toric variety 
$T_{\rm emb}(\Delta)$ defined by a complete fan $\Delta$ in $N \simeq \bZ^d$. We denote by $f_a$ the automorphism of $Y_{\Delta}$ associated to an element
$$a:=(a_1,a_2,\dots,a_d)\in (\bC^*)^d$$ 
under the canonical inclusion $(\bC^*)^d\subset {\rm Aut}\,(Y_\Delta)$. All what we need for toric varieties is covered by the book \cite{Oda}.
\end{notation} 

\begin{theorem}\label{main1} Let $(S, F)$ be a McMullen's pair defined 
in Section 2 and $Y = Y_{\Delta}$ be a $d$-dimensional non-singular projective toric variety. Set $N$ to be the number of the $d$-dimensional cones in $\Delta$. 
Then, 
there is 
$$a = (a_1,a_2,\dots,a_d) \in (\bC^*)^d$$ 
such that 
$$g := (F, f_a) \in\, {\rm Aut}\, (S \times Y)$$ 
satisfies:

(1) $h(g) = h(F) >0$; and 

(2) $g$ has exactly $N$ Siegel disks and they are all arithmetic.
\end{theorem}

Note that toric varieties are always rational. We also note that 
for $d \ge 2$, we can make $N$ as large as we want, while 
for $d=1$, we have $Y = \bP^1$ and $N=2$. So, Theorem (\ref{main}) (1), (2) follows from this theorem. We shall prove Theorem (\ref{main1}) in Section 5.

\begin{theorem}\label{main2} For each given $d$, there are 
$d$ McMullen's pairs
$$(S_1, F_1)\, ,\, (S_2, F_2)\, ,\, \dots\, ,\, 
(S_d, F_d)$$ 
such that 
$$g := (F_1, F_2, \dots , F_d) \in {\rm Aut}\, 
(S_1 \times S_2 \times \cdots \times S_d)$$
satisfies:

(1) $h(g) = h(F_1) + h(F_2) + \cdots + h(F_d) >0$; and 

(2) $g$ has exactly one Siegel disk and it is arithmetic.
\end{theorem}
Theorem (\ref{main}) (3) clearly follows from this theorem. We shall prove 
Theorem (\ref{main2}) at the end of Section 5.

\section{Salem polynomials and multiplicatively independent sequences}
\setcounter{lemma}{0}
\noindent
In this section, we introduce the notion of ``multiplicatively independent sequence of algebraic integers on the unit circle (MAU)'' and show the existence. 
The existence of a MAU is crucial in our product construction. 
\begin{definition}\label{mau}
Let
$$\alpha_1\, ,\, \beta_1\, ,\, \dots\, ,\, \alpha_m\, ,\, \beta_m$$
be a sequence of complex numbers of length $2m$. We call this sequence a
``multiplicatively independent sequence of algebraic integers on the unit circle 
of length $2m$'' (MAU of length $2m$ for short) if the following 
(i), (ii) and (iii) are satisfied:
\begin{enumi}
\item $\alpha_i$, $\beta_i$ ($1 \le i \le m$) are algebraic integers;

\item $\alpha_i$, $\beta_i$ ($1 \le i \le m$) are of absolute value 
$1$, i.e., they are on the unit circle; and

\item $(\alpha_1, \beta_1, \dots, \alpha_m, \beta_m)$ is 
multiplicatively independent. 
\end{enumi}
By abuse of language, we call the following subsequence of a MAU 
of length $2m$,
$$\alpha_1\, ,\, \beta_1\, ,\, \dots\, ,\, \alpha_{m-1}\, ,\, \beta_{m-1}\, 
,\, \alpha_{m}$$
a MAU of length $2m-1$. 
\end{definition}

\begin{theorem}\label{exist}
Any MAU of length $2m$ can be extended to a MAU of length $2(m+1)$. 
In particular, there is an infinite sequence
\begin{equation}\label{*}
\alpha_1\, ,\, \beta_1\, ,\, \dots\, ,\, \alpha_m\, ,\, 
\beta_m\, ,\,\alpha_{m+1}\, ,\, 
\beta_{m+1}\, ,\, \dots
\end{equation}
such that for any given integer $n >0$, the first $n$ terms of this sequence 
form a MAU of length $n$.
\end{theorem}
\begin{remark}\label{explicit} 
In the proof, we shall give an explicit construction of the 
sequence \eqref{*}. This explicit construction is essential in our proof 
of Theorem (\ref{main2}). In fact, in our construction of the sequence 
\eqref{*}, there is a McMullen's pair $(S_m, F_m)$ for each $m$ such that $(S_m)^{F_m} = \{P_m, Q_m\}$ and $F_m$ has an arithmetic Siegel disk at $Q_m$ with
$$F_m^{*}(x, y) = (\alpha_m x, \beta_m y)$$
for appropriate local coordinates $(x, y)$ at $Q_m$ (but no Siegel disk 
at $P_m$). 
\end{remark}
\begin{proof} We construct $\alpha_m, \beta_m$ inductively. 

For each positive integer $k$, we set
$$n(k): = 360k +19\, ,\, d(k) : = 180k +7\, .$$
Then $n(k) \equiv 1\, {\rm mod}\, 6$. Note that $180$ and $7$ are coprime. 
Then, by Dirichlet's Theorem (see e.g. \cite{Se73}, Page 25, Lemma 3),
there are infinitely many prime numbers in the sequence 
$$d(1)\, ,\, d(2)\, ,\, \dots \, ,\, d(k)\, ,\, \dots\, .$$

First we construct a MAU $\alpha_1, \beta_1$ of length $2$. Choose a sufficiently large prime number $p_1 = d(k_1)$. Set $n_1 := n(k_1)$. As $n_1 \equiv 1\, {\rm mod}\, 6$ 
and $n_1$ is also sufficiently large, we can apply Theorem (\ref{Mc}) for 
this $n_1$. Hence we obtain a McMullen's pair 
$$(S_1, F_1) := (S(n_1), F(n_1))$$ 
with a Siegel disk at $Q_1$ such that
$$(F_1)^{*}(x, y) = (\alpha(n_1) x, \beta(n_1) y)\, $$ 
at $Q_1\,$. This is arithmetic by Proposition \eqref{2.7}.
 {\it Here and hereafter, to describe McMullen's pairs, 
we adopt the same notation as in Theorem (\ref{Mc}).} Set
$$\alpha_1 := \alpha(n_1)\, ,\, \beta_1 := \beta(n_1)\, .$$
Then, $\alpha_1$ and $\beta_1$ form a MAU of length $2$. 

Next, assuming that we have constructed a MAU of length $2m$ 
$$\alpha_1\, ,\, \beta_1\, ,\, \dots\, ,\, \alpha_m\, ,\, 
\beta_m\, ,$$
we shall extend this sequence to a MAU of length $2(m+1)$. 

Let us consider the field extension
$$K := \bQ(\alpha_1\, ,\, \beta_1\, ,\, \dots\, ,\, \alpha_m\, ,\, 
\beta_m)\, .$$ 
We put $\ell := [K :\bQ]$. 
Then, choose sufficiently large $k$ such that $q := d(k)$ is a prime number 
with $q > \ell$. Set $n := n(k)$. 
As $n \equiv 1\, {\rm mod}\, 6$, we can apply Theorem (\ref{Mc}) 
for this $n$. Then, we obtain a McMullen's pair 
$(S(n),F(n))$ 
with important values $\delta(n)$, $\alpha(n)$, $\beta(n)$, $\delta'(n)$, 
$\alpha'(n)$, $\beta'(n)$ and the Salem polynomial $\varphi (x)$ as described in Theorem (\ref{Mc}). 
We set: 
$$\delta := \delta(n)\, ,\, \delta' := \delta'(n)\, ,\, \alpha := 
\alpha(n)\, ,\, \beta := \beta(n)\, ,\, \alpha' := 
\alpha'(n)\, ,\, \beta' := \beta'(n)\, .$$

We shall show that 
\begin{equation}\label{**}
\alpha_1\, ,\, \beta_1\, ,\, \dots\, ,\, \alpha_m\, ,\, 
\beta_m\, ,\, \alpha\, ,\, \beta
\end{equation}
is a MAU of length $2(m+1)$. {\it We then define} $\alpha_{m+1} := \alpha$ 
and $\beta_{m+1} := \beta$, {\it and continue}. 

By the assumption (on the first $2m$ terms) 
and by Theorem (\ref{Mc})(3) and Proposition (\ref{2.7}), we already know that each term of \eqref{**} 
is an algebraic integer of absolute value $1$. Thus, it suffices to show 
that they are multiplicatively independent. We shall prove this from now.

First we compute the degree of the Salem polynomial $\varphi(x)$ and its Salem trace polynomial 
$r(x)$. By  
$$n = 360k +19 \equiv 19\, {\rm mod}\, 360\, ,$$
we have $C_{n}(x) = C_{19}(x)$ in Theorem (\ref{Mc})(1) (cf. 
Theorem (\ref{cox})). 
On the other hand, by Example (\ref{19}), we have 
${\rm deg}\, C_{19}(x) = 5$. Thus
$${\rm deg}\, \varphi(x) = 360k+19 -5 = 360k +14\, .$$ Hence 
$${\rm deg}\, r(x) = \frac{{\rm deg}\, \varphi(x)}{2} 
= 180k +7 = d(k) = q\, .$$ 

As $r(x)$ is irreducible over $\bZ$, it follows that 
$$[\bQ(\delta + \frac{1}{\delta}) : \bQ] = q\, .$$
As $q > \ell$ and $q$ is a prime number, 
we have then that
$$[K(\delta + \frac{1}{\delta}) : K] = q\, .$$
So, $r(x)$ is also irreducible over $K$. Let $L$ be the Galois closure 
of $K(\delta + \frac{1}{\delta})$ in the algebraic closure $\overline{K}$. 
As 
$$\delta + \frac{1}{\delta}\, ,\, 
\delta' + \frac{1}{\delta'}\, ,\, \eta + \frac{1}{\eta}$$
are roots of $r(x) = 0$, there are $\sigma \in {\rm Gal}(L/K)$ 
and $\tau \in {\rm Gal}(L/K)$ such that 
$$\sigma(\delta + \frac{1}{\delta}) = \delta' + \frac{1}{\delta'}\, ,\, 
\tau(\delta + \frac{1}{\delta}) = \eta + \frac{1}{\eta}\, .$$
Here $\eta$ is the Salem number of $\varphi(x)$.
Extending $\sigma$ and $\tau$ to ${\rm Gal}\, (\overline{K}/K)$, we have 
$$\sigma(\delta) = \delta'\,\, {\rm or}\,\, 
 \sigma(\delta) = \frac{1}{\delta'} 
= \overline{\delta'}\, ,$$
$$\tau(\delta) = \eta\,\, {\rm or}\,\, \tau(\delta) = \frac{1}{\eta}\, .$$
Let 
\begin{equation}\label{***}
\alpha_1^{\ell_1}\beta_1^{k_1} \cdots \alpha_m^{\ell_m}\beta_m^{k_m} 
\alpha^{\ell_{m+1}}\beta^{k_{m+1}} = 1
\end{equation}
where $\ell_i$, $m_i$ are integers. Transforming \eqref{***} by $\sigma$ 
(and switching $\alpha'$ and $\beta'$ if necessary), 
we obtain either 
$$\alpha_1^{\ell_1}\beta_1^{k_1} \cdots \alpha_m^{\ell_m}\beta_m^{k_m} 
(\alpha')^{\ell_{m+1}}(\beta')^{k_{m+1}} = \pm 1\,\, {\rm or}$$
$$\alpha_1^{\ell_1}\beta_1^{k_1} \cdots \alpha_m^{\ell_m}\beta_m^{k_m} 
(\overline{\alpha'})^{\ell_{m+1}}(\overline{\beta '})^{k_{m+1}} = \pm 1\,\, .$$
Here we use Theorem (\ref{Mc})(3) and (4). 
In the first case, taking (the square of) the norm, we get 
$$1 = \vert (\alpha')^{2\ell_{m+1}}(\beta')^{2k_{m+1}} \vert 
= \vert (\alpha'/\beta')^{\ell_{m+1} - k_{m+1}}(\alpha'\beta')^{\ell_{m+1} 
+ k_{m+1}} \vert = \vert \alpha'/\beta' \vert^{\ell_{m+1} - k_{m+1}}\, .$$
Here we use $\vert \alpha' \beta'\vert = \vert \delta' \vert = 1$ (Theorem 
(\ref{Mc})(4)). As $\vert \alpha'/\beta' \vert \not= 1$ 
(Theorem (\ref{Mc})(4)), 
it follows that 
$$\ell_{m+1} - k_{m+1} = 0\, .$$
For the same reason, this is true also for the second case. 
Substituting this into \eqref{***}, and using $\alpha \beta = \delta$, we obtain
$$\alpha_1^{\ell_1}\beta_1^{k_1} \cdots \alpha_m^{\ell_m}\beta_m^{k_m} 
\delta^{k_{m+1}} = 1\, .$$
Transforming this equality by $\tau$, we get either 
$$\alpha_1^{\ell_1}\beta_1^{k_1} \cdots \alpha_m^{\ell_m}\beta_m^{k_m} 
\eta^{k_{m+1}} = 1\,\, \, {\rm or}\,\, \, 
\alpha_1^{\ell_1}\beta_1^{k_1} \cdots \alpha_m^{\ell_m}\beta_m^{k_m} 
(\frac{1}{\eta})^{k_{m+1}} = 1\, .$$ 
Taking the norm, we get
$$\eta^{k_{m+1}} = 1\, .$$
As $\eta >1$, this implies $k_{m+1} = 0$. Thus 
$$\ell_{m+1} = k_{m+1} = 0\, .$$
Substituting this into \eqref{***}, we obtain
$$\alpha_1^{\ell_1}\beta_1^{k_1} \cdots \alpha_m^{\ell_m}\beta_m^{k_m} 
 = 1\,\, .$$
As $\alpha_1, \beta_1, \cdots , \alpha_{m}, \beta_{m}$ is a MAU of length $2m$, 
it follows that 
$$\ell_1 = k_1 = \cdots = \ell_m = k_m = 0\, .$$
This completes the proof. 
\end{proof}

\section{Proof of main Theorems}
\setcounter{lemma}{0}
\noindent
Theorem (\ref{main1})(1) follows from the next two Propositions 
(\ref{proposition 5.1}), (\ref{proposition 5.2}):

\begin{proposition}\label{proposition 5.1}
Let $Y$ be a non-singular projective toric variety. Then, for any
$f \in {\rm Aut}\,(Y)$ (not necessarily in the big torus), we have $h(f)=0$.
\end{proposition}
\begin{proof}
Recall that the cone $\overline{\rm NE}\,(Y)$ of numerically effective curves of $Y$ is finite, rational and 
polyhedral (see e.g. \cite{Oda}, Page 107, Proposition 2.26).
Thus, the ample cone $A(Y)\subset H^2(Y,\bR)$ is also  finite, rational and polyhedral,
as it is the dual cone of $\overline{\rm NE}\,(Y)$ (Kleiman's criterion).

Let $f \in {\rm Aut}\,(Y)$ and consider the induced action $f^*$ on $H^2(Y,\bZ)$. Let 
$L_i$  ($1\leq i\leq \ell$) be the $1$-dimensional edges of $\overline{A(Y)}$
and $v_i$ be the primitive vector of $L_i$. Then,
$(f^*)^{\ell !}$ is the identity on $\{ L_1,\dots,L_\ell\}$, hence $(f^*)^{\ell !}(v_i)=v_i$ ($1\leq i \leq \ell$).
As $A(Y)$ is open in $H^2(Y,\bR)$ (because $h^{2,0}(Y)=h^{0,2}(Y)=0$), it follows that
$(f^*)^{\ell !}$ is the identity on $H^2(Y,\bR)$. Hence $\rho\left( f^*| H^2(Y,\bZ)\right)=1$.
Thus, by \cite{DS04} Corollaire (2.2), we have $h(f)=0$.
\end{proof}

\begin{remark}
It is known that ${\rm Aut}\,(Y)$ is generated by three classes of automorphisms: the torus, roots and fan symmetries (see
e.g. \cite{CK99}, Page 48, Theorem 3.6.1). Among these three classes, the torus and the roots are in 
${\rm Aut}^0(Y)$, the identity component of ${\rm Aut}\,(Y)$. So, their action on $H^*(Y,\bZ)$
are the identity. In our construction this is enough, but 
Proposition (\ref{proposition 5.1}) also follows from this description.
\end{remark}

\begin{proposition}\label{proposition 5.2}
Let $Y$ be a non-singular projective toric variety, $f\in{\rm Aut}(Y)$ and
 $F \in {\rm Aut}(S)$ be an automorphism of a compact K\"ahler manifold $S$.
Let 
$$g\, :=\, (F,f)\, \in\,\, {\rm Aut}\, (S\times Y)\, .$$ 
Then $h(g)=h(F)$. In particular,
$g$ is of positive entropy if and only if so is $F$.
\end{proposition}
\begin{proof} Recall that $H^i(Y,\bZ)=0$ for $i$ odd by Jurkiewicz-Danilov's Theorem (see
e.g. \cite{Oda}, Page 134). Then, by K\"unneth formula:
$$
H^{2k}(S\times Y,\bQ)= \oplus_{\ell=0}^k H^{2\ell}(S,\bQ)\otimes 
H^{2k-2\ell}(Y,\bQ).
$$
Here $g^* =F^* \otimes f^*$ on each direct summand. By Proposition (\ref{proposition 5.1}),
we have that $h(f)=0$. Thus, the eigenvalues of $f^*$ on $H^{2k-2\ell}(Y,\bQ)$ are of 
absolute value $1$. In fact, letting $\epsilon_i$ ($1\leq i \leq a$) be the eigenvalues of
$f^*|H^{2k-2\ell}(Y,\bQ)$ counted with multiplicities, then $|\epsilon_i|\leq 1$ ($1\leq i \leq a$).
But $\det f^*=\pm 1$ as $f^*$ is an automorphism of $H^{2k-2\ell}(Y,\bZ)$, so 
$$|\prod_{i=1}^a \epsilon_i|=1\, $$
hence $|\epsilon_i|=1$ ($1\leq i \leq a$). Thus 
$$\rho\left( g^*| H^{2\ell}(S,\bQ)\otimes H^{2k-2\ell}(Y,\bQ)\right)
=\rho\left( F^*| H^{2\ell}(S,\bQ)\right)\, .$$ 
This implies the result.
\end{proof}

We now prove Theorem (\ref{main1})(2). Let $\{\sigma_1,\dots,\sigma_N\}$ be the set
of $d$-dimensional cones of $\Delta$. Then $Y=\cup_{p=1}^N U_p$, where $U_p:={\rm Spec}\bC[\sigma_p^\vee \cap M]$.
As we assume $Y$ is non-singular, each $\bC[\sigma_p^\vee \cap M]$ is 
written as
$$
\bC[\sigma_p^\vee \cap M]=\bC[x^{K_1(p)},x^{K_2(p)},\dots,x^{K_d(p)}]\subset \bC[M]=\bC[x_1^{\pm 1},x_2^{\pm 1},\dots,x_d^{\pm 1}]
$$
and $U_p\cong \AN^d$ with the coordinates $(x^{K_i(p)})$. Here we use multi-index notation, namely
$$
K_i(p)=(k_{i1}(p),k_{i2}(p),\dots,k_{id}(p))\in \bZ^d 
$$
and
$$
 x^{K_i(p)}=x_1^{k_{i1}(p)}x_2^{k_{i2}(p)}\cdots x_d^{k_{id}(p)}.
$$
Let now $a:=(a_1,a_2,\dots,a_d)\in (\bC^*)^d$ and let $f_a \in {\rm Aut}\,(Y)$ be the corresponding automorphism.
 Each $U_p$ is invariant
under $f_a$ and the action of $f_a$ on $U_p$ is given as follows (we use again multi-index notation)
\begin{equation}\label{formula 5.1}
f_a^*(x^{K_1(p)},x^{K_2(p)},\dots,x^{K_d(p)})=(a^{K_1(p)}x^{K_1(p)},a^{K_2(p)}x^{K_2(p)},\dots,a^{K_d(p)}x^{K_d(p)}).
\end{equation}
We have the following 
\begin{lemma}\label{lemma 5.3}
Let $a=(a_1,a_2,\dots,a_d)\in (S^1)^d$ be multiplicatively independent.
Then $f_a\in  {\rm Aut}\,(Y)$ has exactly $N$ fixed points on $Y$.
\end{lemma}
\begin{proof}
The set of $d$-dimensional cones $\{\sigma_i\}_{i=1}^N$ bijectively corresponds to the set of $0$-dimensional 
orbits of $(\bC^*)^d$ (each of which is clearly one point), say 
$$Q_1\, ,\, \dots\, ,\, Q_N$$ 
(see e.g. \cite{Oda} Page 10, Proposition 1.6).
As $f_a\in {\rm Aut}(Y)$ is associated to $a\in(\bC^*)^d$, it follows 
that $f_a(Q_p)=Q_p$, therefore $f_a$ has at least $N$ fixed points.

On the other hand, as $a$ is multiplicatively independent, $f_a$ has exactly one fixed point on each $U_p$ ($1\leq p\leq N$),
 that is the origin of $U_p\cong \AN^d$ under the coordinates  $(x^{K_i(p)})$. This follows directly from formula \eqref{formula 5.1}.
Hence $f_a$ has exactly $N$ fixed points $Q_1,\dots,Q_N$.
\end{proof} 

To conclude the proof of Theorem (\ref{main1}) (2), we take a McMullen's pair 
$(S,F)$. Hence $S^F$, the set of fixed points of $F$, consists of two points $P$ and $Q$ but
$F$ admits an arithmetic Siegel disk only at $Q$. This means that there are 
analytic coordinates $(z_1,z_2)$ at $Q$
such that 
$$F^*(z_1,z_2)=(b_1z_1,b_2z_2)\, ,$$
where $(b_1,b_2)\in (S^1)^2$ is a MAU of length $2$ (see Section 3 for the notation).
Now we use Theorem (\ref{exist}) to find 
$$a_1\, ,\, \dots\, ,\, a_d\, \in S^1$$ such that 
$$b_1\, ,\, b_2\, ,\, a_1\, ,\, \dots\, ,\, a_d$$ 
is a MAU of length $d+2$.
Then, set 
$$g\,:=\,(F,f_a)\, \in\,\, {\rm Aut}\,(S\times Y)\, .$$ 
It follows from Lemma (\ref{lemma 5.3}) that $g$ has exactly $2N$ fixed points
$$(P,Q_1)\, ,\, \dots\, ,\,(P,Q_N)\, ,\, (Q,Q_1)\, ,\, \dots\, ,\, 
(Q,Q_N)\, .$$ However, the first $N$ of these have no Siegel disk. 
Let us show that the last $N$ of these, namely
$$(Q,Q_1)\, ,\,\dots\, ,\,(Q,Q_N)\, ,$$ 
have arithmetic Siegel disks. Using notation as in  Formula 
\eqref{formula 5.1}, we have local coordinates
$(y_i)_{i=1}^d:=(x^{K_i(p)})_{i=1}^d$ at $Q_p\in Y$ such that 
$$
f_a^*(y_1\, ,\, y_2\, ,\, \dots\, ,\, y_d)=(a^{K_1(p)}y_1\, ,\, 
a^{K_2(p)}y_2\, ,\, \dots\, ,\, a^{K_d(p)}y_d)\, .
$$
So, under the local coordinates $(z_1,z_2,y_1,y_2,\dots,y_d)$ on $S\times Y$,
the action of $(F,f_a)$ at $(Q,Q_p)$ is linearized as 
$$
g^*(z_1\, ,\, z_2\, ,\, y_1\, ,\, y_2\, ,\dots\, , y_d)=(b_1z_1,b_2z_2,a^{K_1(p)}y_1,a^{K_2(p)}y_2,\dots\, , a^{K_d(p)}y_d) \, .
$$ 
It remains to prove that 
$$b_1\, ,\, b_2\, ,\,a^{K_1(p)}\, ,\, a^{K_2(p)}\, ,\, \dots\, ,\, a^{K_d(p)}$$ form a MAU. By construction
and by Proposition \eqref{2.7},
each term is an algebraic integer 
of absolute value $1$. Let us show that they are multiplicatively independent. 
Hence take 
$$(\ell_1\, ,\, \ell_2\, ,\, r_1\, ,\, r_2\, ,\, \dots\, ,\, r_d)\, 
\in\, \bZ^{d+2}$$ such that 
\begin{equation}\label{formula 5.2}
b_1^{\ell_1}b_2^{\ell_2}(a^{K_1(p)})^{r_1}(a^{K_2(p)})^{r_2}\cdots (a^{K_d(p)})^{r_d}=1.
\end{equation}
If $(r_1,r_2,\dots,r_d)=(0,0,\dots,0)$, then $\ell_1=\ell_2=0$ since $(b_1,b_2)$ is multiplicatively independent.
Therefore we can assume  $(r_1,r_2,\dots,r_d)\not=(0,0,\dots,0)$.  Then 
$$
b_1^{\ell_1}b_2^{\ell_2}a^{s_1}_1a^{s_2}_2\cdots a^{s_d}_d=1,
$$ 
where 
$$
(s_1,s_2,\dots,s_d)=(r_1,r_2,\dots,r_d)\cdot \left( \begin{matrix}
k_{11}(p)&k_{12}(p)&\dots&k_{1d}(p)\\
k_{21}(p)&k_{22}(p)&\dots&k_{2d}(p)\\
\dots    &\dots    &\dots&\dots    \\
k_{d1}(p)&k_{d2}(p)&\dots&k_{dd}(p)
\end{matrix}\right).
$$
Since $Y$ is complete and non-singular, the primitive vectors of the $1$-dimensional rays of $\sigma_p$
form a $\bZ$-basis of $N\cong\bZ^d$ (see e.g. \cite{Oda} Page 15, Theorem 1.10). Thus the row vectors of the previous matrix,
which generate the dual cone  $\sigma_p^\vee\subset M_{\bR}$, form a $\bZ$-basis of $M$ as well.
We conclude that $(s_1,s_2,\dots,s_d)\not=(0,0,\dots,0)$ and hence we get a contradiction since the sequence
$$b_1\, ,\, b_2\, ,\, a_1\, ,\, a_2\, ,\, \dots\, ,\, a_d$$ is a MAU, hence multiplicatively independent.

This concludes the proof of Theorem (\ref{main1}) (2). 
\par
\vskip 4pt
Let us prove Theorem (\ref{main2}). 
We take a sequence of McMullen's pairs
$$(S_1\, ,\, F_1)\, , \, \dots\, ,\, (S_d,F_d)$$ 
as in Remark (\ref{explicit}). 
Set
$$
g:=(F_1,\dots,F_d)\in {\rm Aut}\, (S_1\times\cdots \times S_d) \, .
$$
As in the proof of Proposition (\ref{proposition 5.2}), by K\"unneth decomposition,
we obtain Theorem (\ref{main2}) (1).

We notice  that
$$(S_1\times\cdots \times S_d)^{g}=\{R = (R_1,\dots,R_d)\,|\, R_i\in \{P_i,Q_i\}, 1\leq i\leq d\}\, .$$
By construction, we have an arithmetic Siegel disk 
at 
$$(Q_1\, ,\, Q_2\, , \dots\, , Q_d)$$ 
but, if $R_i=P_i$ for some $i$, then we have no Siegel disk at $R$ 
by Theorem (\ref{Mc})(3). This completes the proof of Theorem (\ref{main2}). 

\par
\vskip 4pt
{\it Acknowledgements.} We would like to express our thanks to Professors Fabrizio Catanese and Thomas Peternell for their warm encouragement. We are grateful to the renewed program of Alexander von Humboldt Foundation and DFG Forschergruppe 790 ``Classification of algebraic surfaces and compact complex manifolds'', which make our collaboration possible.


\vskip .2cm \noindent
Keiji Oguiso \\
Department of Mathematics, Osaka University\\
Toyonaka 560-0043 Osaka, Japan\\
oguiso@math.sci.osaka-u.ac.jp

\vskip .2cm \noindent
Fabio Perroni \\
Lehrstuhl Mathematik VIII, Mathematisches Institut\\
Universit\"at Bayreuth, D-95440 Bayreuth, Germany\\
fabio.perroni@uni-bayreuth.de 

\end{document}